\documentclass[12]{amsart}
\usepackage{amssymb,amsfonts,latexsym}

\newtheorem{theorem}{Theorem}[section]
\newtheorem{lemma}[theorem]{Lemma}

\newtheorem{corollary}[theorem]{Corollary}

\theoremstyle{definition}
\newtheorem{definition}[theorem]{Definition}

\newtheorem{conjecture}[theorem]{Conjecture}
\newtheorem{remark}[theorem]{Remark}
\newtheorem{general remarks}[theorem]{General remarks}

\newcommand{\id}{\operatorname{id}}

\renewcommand{\span}{\operatorname{span}}

\newcommand{\ben}{\begin{enumerate}}
\newcommand{\een}{\end{enumerate}}

\begin{document}
\title[Group graded PI-algebras and their codimension growth] {Group graded PI-algebras and their codimension growth}

\author{Eli Aljadeff}
\address{Department of Mathematics, Technion-Israel Institute of
Technology, Haifa 32000, Israel}
\email{aljadeff@tx.technion.ac.il}

\date{Feb. 6, 2010}

%\begin{abstract}

%\end{abstract}


\thanks {The author was partially supported by the ISRAEL SCIENCE FOUNDATION
(grant No. 1283/08) and by the E.SCHAVER RESEARCH FUND}

%The second author was partially supported by the ISRAEL SCIENCE
%FOUNDATION (grant No. 1178/06) and also by the Russian Fund of
%Fundamental Research, grant  $\ RFBR 08-01-91300-IND_a$.}

\begin{abstract} Let $W$ be an associative \textit{PI}-
algebra over a field $F$ of characteristic zero. Suppose $W$ is
$G$-graded where $G$ is a finite group. Let $\exp(W)$ and
$\exp(W_{e})$ denote the codimension growth of $W$ and of the
identity component $W_{e}$, respectively. The following inequality
had been conjectured by Bahturin and Zaicev: $\exp(W)\leq |G|^2
\exp(W_{e}).$ The inequality is known in case the algebra $W$ is
affine (i.e. finitely generated). Here we prove the conjecture in
general.
\end{abstract}

\maketitle

\begin{section}{Introduction} \label{Introduction}

Let $W$ be any $PI$-algebra over a field of characteristic zero.
Suppose $W$ is $G$-graded where $G$ is any finite group. Let
$\exp(W)$ and $\exp(W_{e})$ be the exponents of the algebra $W$
and of its $e$-component (determined by the $G$-grading). The main
objective of this paper is to prove the following conjecture (see
\cite{BZ}):

\begin{conjecture} \label{Bahturin-Zaicev-Conjecture}
$$
\exp(W) \leq \mid G\mid^{2} \exp(W_e).
$$

\end{conjecture}

The case where $W$ is affine was proved in \cite{Al}. Here we show
that the conjecture holds also when $W$ is non affine.

\begin{theorem} \label{Main theorem} Let $W$ be as above. Then $\exp(W) \leq |G|^2 \exp(W_e).$
\end{theorem}

The proof is based on three ingredients. The first is the
representability theorem for $G$-graded $PI$-algebras which allows
us to ``pass" from arbitrary algebras to finite dimensional
algebras. The second is a precise description of the structure of a
$G$-graded simple algebra (as determined by Bahturin, Sehgal and
Zaicev) and the third is the relation (established by Giambruno and
Zaicev) between the $\exp(W)$ and the structure of $W$ when $W$ is a
finite dimensional algebra. In fact these ingredients were also used
in the proof of the conjecture in the affine case. The point here is
that when passing to finite dimensional algebras (with the
representability theorem for non-affine $G$-graded algebras), we are
``forced" to consider some statements which are \textit{more
general} than Bahturin-Zaicev conjecture. In section $1$ (after some
preliminaries) we formulate these statements and in particular
translate the problem to finite dimensional algebras. In section $2$
we solve the problem for finite dimensional algebras in sufficient
generality to imply Bahturin-Zaicev conjecture for non affine
algebras.

\end{section}

\begin{section}{Preliminaries and translation of the problem to finite dimensional algebras}

We start this section by recalling some definitions. Let $X$ be a
countable set of indeterminates and $F\langle X\rangle$ the
corresponding free algebra over $F$. Let $\id(W)$ be the $T$-ideal
of identities of $W$ in $F\langle X\rangle$ and let $\mathcal{W}=
F\langle X\rangle/\id(W)$ be the relatively free algebra of $W$. We
denote by $c_n(W)$ the dimension of the subspace spanned by
multilinear elements in $n$ free generators of $\mathcal{W}$. We
refer to $c_n(W)$ as the $n-th$ term of the codimension sequence of
the algebra $W$. Our interest is in the (exponential) asymptotic
behavior of the sequence of codimensions, namely in
$$
\exp(W)=\lim_{n\to \infty}\sqrt[n]{c_n(W)}.
$$

It is known that $\exp(W)$ exists and moreover it assumes only
integer values (see~\cite{gz1},~\cite{gz2},~\cite{reg2}).

If the algebra~$W$ is $G$-graded, where $G$ is any group, we may
consider $G$-graded polynomial identities and the corresponding
$G$-graded exponent. Since our main theorem will be formulated in
these terms, we recall them here as well. Let $X_G = \cup_{g\in
G}X_{g}$ where $X_g = \{x_{1,g}, x_{2,g},\ldots \}$ is a countable
set of indeterminates and let $F\langle X_G \rangle$ be the free
$G$-graded algebra on the set $X_G$ (here, a monomial
$x_{i_1,g_{i_1}}x_{i_2,g_{i_2}}\cdots x_{i_r,g_{i_r}} $ is
homogeneous of degree $g=g_{i_1}g_{i_2}\cdots g_{i_r}$). In
$F\langle X_G \rangle$ we consider $\id_{G}(W)$, the ideal of
$G$-graded identities of $W$, and denote by
$\mathcal{W}_{G}=F\langle X_G \rangle/ \id_{G}(W)$ the corresponding
relatively free $G$-graded algebra. For $n=1, 2,\ldots,$ we denote
by $c^G_n(W)$ the dimension of the subspace of $\mathcal{W}_{G}$
spanned by multilinear elements in $n$ ($G$-graded) free generators.
We call $\{c^G_n(W)\}_{n=1}^{\infty}$ the sequence of
$G$-codimensions of $W$. The $G$-graded exponent of $W$ is given
(when it exists) by

$$
\exp_{G}(W)=\lim_{n\to \infty}\sqrt[n]{c^G_n(W)}.
$$

\begin{remark}
The limit $\lim_{n\to \infty}\sqrt[n]{c^G_n(W)}$ is known to exist
when $G$ is abelian (see~\cite{AGM} and \cite{AB} in case $W$ is
affine and \cite{GM} in case $W$ is non affine). Moreover in that
case it assumes only integer values. In the general case (i.e. when
$G$ is not necessarily abelian) it is known only that the sequence
of codimensions $c^G_n(A)$ is exponetially bounded.
\end{remark}

As mentioned in the introduction, we wish to translate the problem
to the ``world" of finite dimensional algebras. Before doing it
let us recall two results, the first of which is Kemer's
representability theorem for arbitrary \textit{PI}-algebras (see
\cite{Kemer1}, \cite{Kemer2}, \cite{Kemer3}).

\begin{theorem}[\textit{PI}-equivalence]\label{PI-equivalence-general}
Let $W$ be a \textit{PI}-algebra over $F$. Then there exists a
field extension~$L$ of $F$ and a finite dimensional
$\mathbb{Z}/2\mathbb{Z}$-graded algebra $A$ over $L$ such that
$\id(W)=\id(A^{*})$ where $A^{*}$ is the Grassmann envelope of
$A$.
\end{theorem}

Next we recall a result of Giambruno and Zaicev (see \cite{gz2})
which describes the (exponential) codimension growth of an algebra
$W$ in terms of the structure of the $\mathbb{Z}/2\mathbb{Z}$-graded
algebra $A$ appearing in Kemer's result.

\begin{theorem} \label{Giam-Zaicev}Let $W$ be an arbitrary
\textit{PI}-algebra over a field of characteristic zero $F$. Let $A$
be a $\mathbb{Z}/2\mathbb{Z}$-graded algebra as in Kemer's theorem.
Then $\exp(W)=\exp_{\mathbb{Z}/2\mathbb{Z}}(A)$.
\end{theorem}

\begin{remark} \label{formulation of Giambruno_Zaicev}
In fact, Theorem \ref{Giam-Zaicev} was stated in \cite{gz2} in a
different way. The result there says that $\exp(W)$ is equal to the
dimension of a certain semisimple subalgebra of $\overline{A}$,
where $\overline{A}=A\otimes_{L}\bar{L}$ and $\bar{L}$ is the
algebraic closure of $L$. Then one finds that this dimension is
actually equal to $\exp_{\mathbb{Z}/2\mathbb{Z}}(A)$ (see \cite
{AGM}).
\end{remark}

The other ``player'' in Conjecture \ref{Bahturin-Zaicev-Conjecture}
is of course $\exp(W_e)$. We could describe it in similar terms as
we did for $\exp(W)$, i.e. using a $\mathbb{Z}/2\mathbb{Z}$-graded
algebra $A'$ say, but by doing it we won't be able to compare
$\exp_{\mathbb{Z}/2\mathbb{Z}}(A)$ and
$\exp_{\mathbb{Z}/2\mathbb{Z}}(A')$. We need therefore a $G$-graded
representability theorem which will produce the
$\mathbb{Z}/2\mathbb{Z}$-graded algebras $A$ and $A'$
simultaneously. This is achieved by invoking the $G$-graded
representability theorem for arbitrary $G$-graded
\textit{PI}-algebras (see \cite{AB}).

\begin{theorem}[$G$-graded \textit{PI}-equivalence]\label{PI-equivalence-general}

Let $W$ be a \textit{PI} and $G$-graded algebra over $F$ where $G$ is any finite group. Then there
exists a field extension $L$ of $F$ and a finite dimensional
$\mathbb{Z}/2\mathbb{Z} \times G$-graded algebra $A$ over $L$ such
that $\id_{G}(W)=\id_{G}(A^{*})$ where $A^{*}$ is the Grassmann
envelope of $A$.
\end{theorem}

Let us draw some consequences from Theorem
\ref{PI-equivalence-general} ($W$ and $A^{*}$ are as in the
theorem).

\begin{enumerate}

\item

$\id(W)=\id(A^{*})$ (as ungraded algebras). Indeed, by linearity, any ordinary polynomial identity
is equivalent to a set of $G$-graded identities.

\item
$\id(W_e)=\id((A^{*})_e)$

\item

Note that $A_{\mathbb{Z}/2\mathbb{Z} \times {e}}$, the
$\mathbb{Z}/2\mathbb{Z} \times {e}$-component of $A$, is a
$\mathbb{Z}/2\mathbb{Z}$-graded algebra and we may consider its Grassmann
envelope $(A_{\mathbb{Z}/2\mathbb{Z} \times {e}})^{*}$. It is
easily seen that $(A_{\mathbb{Z}/2\mathbb{Z} \times
{e}})^{*}=(A^{*})_e.$

\end{enumerate}

In view of Theorem \ref{PI-equivalence-general} and its consequences we conclude that
Conjecture \ref{Bahturin-Zaicev-Conjecture} will be proved if we prove the following statement.

\begin{theorem} \label{exp of graded algebra}

Let $A$ be a finite
dimensional $\mathbb{Z}/2\mathbb{Z} \times G$-graded algebra. Let
$A^{*}$ be its Grassmann envelope (with respect to the
$\mathbb{Z}/2\mathbb{Z}$-grading) and let $(A^{*})_e$ be its $e$-component. Then we have

$$
\exp(A^{*})\leq |G|^2 \exp((A_{\mathbb{Z}/2\mathbb{Z} \times
{e}})^{*}).
$$

\end{theorem}

Invoking Theorem \ref{Giam-Zaicev} we see that Theorem \ref{exp of graded algebra} can be replaced by the
following theorem.

\begin{theorem} \label{Theorem after Giam_Zaicev}
With the above notation:

$$
\exp_{\mathbb{Z}/2\mathbb{Z}}(A) \leq |G|^2 \exp_{\mathbb{Z}/2\mathbb{Z}}(A_{\mathbb{Z}/2\mathbb{Z} \times {e}}).
$$

\end{theorem}

As mentioned in Remark \ref{formulation of Giambruno_Zaicev}, if
$F$ is algebraically closed, then the
$\mathbb{Z}/2\mathbb{Z}$-exponent of a finite dimensional
$\mathbb{Z}/2\mathbb{Z}$-graded algebra is determined in terms of
the dimension of a certain semisimple subalgebra of $A$. More
generally, this is known to be true for finite dimensional
$G$-graded algebras where $G$ is a finite abelian group. In
\cite{Al} it was conjectured that this is the case when $G$ is any
finite group. For the reader convenience we recall here the
precise construction in the general case, namely where $G$ is any
finite group.

Let $A$ be a finite dimensional algebra over an algebraically closed
field $F$ of characteristic zero. Suppose $A$ is $G$-graded where
$G$ is a finite group. It is well known (see~ \cite{CM}) that the
radical $J$ of $A$ is $G$-graded as well. Moreover there exists a
$G$-graded semisimple subalgebra $S_{G}$ of $A$ which supplements
$J$ (in $A$) as vector spaces (see~ \cite{StVan}). Let
$S_{G}=(S_{G})_1\oplus \ldots\oplus (S_{G})_q$ be the decomposition
of $S_{G}$ into its $G$-simple components and consider
\textit{nonzero} products in $A$ of the form
$(S_{G})_{i_1}J(S_{G})_{i_2}\cdots J(S_{G})_{i_k}$. For any product
of that form we consider the sum of the dimensions (as $F$-vector
spaces) of the \textit{different} $G$-simple components that appear
in the product. It was conjectured in \cite{Al} that the exponent of
the algebra $A$ as a $G$-graded algebra, is equal to the maximum of
these sums. Let us denote this maximum by $\exp_{G}^{Conj}(A)$.

The following lemma is clear.

\begin{lemma}\label{monotonicity} Let $G$ be a finite group
and $H$ a normal subgroup. Let $A$ be a finite dimensional algebra
over an algebraically closed field $F$ of characteristic zero graded
by $G$. Then $A$ is graded by $G/H$ in a natural way and we have
$\exp_{G}^{Conj}(A) \geq \exp_{G/H}^{Conj}(A)$.
\end{lemma}

Let us pause for a moment and explain what we have so far.
Bahturin-Zaicev conjecture (Conjecture \ref{Bahturin-Zaicev-Conjecture})
for an algebra $W$ over a field $F$ is
given in terms of the growth of the codimensions of the
ideals of identities of $W$ and $W_{e}$. Applying representability for
$G$-graded algebras, namely Theorem \ref{PI-equivalence-general},
the conjecture is translated with the same terms (namely growth of
codimensions) into finite dimensional algebras. Giambruno-Zaicev
result translates the problem in terms of the \textit{structure} of
the finite dimensional algebras but for this the field $F$ is
required to be algebraically closed. In order to fill this ``gap" we
need to pass to the algebraic closure $\bar{F}$ of $F$. Indeed, it
is well known that the ideal of identities (defined over $F$) of $W$
and $W\otimes_{F}L$ coincide where $L$ is any field extension of $F$
and in particular $\bar{F}$. Of course, the algebra $W\otimes_{F}L$
is an $L$-algebra and we may consider its identities over $L$. The
point is that since the field is infinite, these are linear
combinations with coefficient in $L$ of $F$-identities and hence by
changing the ground field from $F$ to $L$ we do not change the
codimensions.

It follows from Lemma \ref{monotonicity} and the last paragraph
that Theorem \ref{Theorem after Giam_Zaicev} (and hence Conjecture
\ref{Bahturin-Zaicev-Conjecture}) will be proved if we show the
following theorem.

\begin{theorem}
Let $A$ be a $\mathbb{Z}/2\mathbb{Z}\times G$-graded finite
dimensional over an algebraically closed field $F$ of
characteristic zero. Then
$$
\exp^{Conj}_{\mathbb{Z}/2\mathbb{Z}\times G}(A) \leq |G|^2 \exp^{Conj}_{\mathbb{Z}/2\mathbb{Z}}(A_{\mathbb{Z}/2\mathbb{Z} \times {e}})
$$

\end{theorem}

Note that since the group $\mathbb{Z}/2\mathbb{Z}$ is abelian, by
the paragraph following Theorem \ref{Theorem after Giam_Zaicev}, the
word ``Conj" could be omitted from the right hand expression.

Our aim is to prove a substantially generalization
of the last statement.

Let $A$ be a finite dimensional algebra over an algebraically closed
field $F$ of characteristic zero. Let $G$ be a finite group and
assume $A$ is $G$-graded. Let $K$ be any subgroup of $G$ and let
$A_{K}$ be the subalgebra of $A$ which corresponds to $K$ (viewed as
a $K$-graded algebra). The following theorem is the main result of
this paper.

\begin{theorem}\label{Main} Let $A$, $G$ and $K$ as above. Then

$$
\exp^{Conj}_{G}(A) \leq [G:K]^2 \exp^{Conj}_{K}(A_{K}).
$$
\end{theorem}

Note that the case where $K=\{e\}$ was proved in \cite{Al}.

\end{section}

\begin{section}{Proof of Theorem \ref{Main}}

As in \cite{Al}, the key here is a result of Bahturin, Sehgal and
Zaicev in which they fully describe the structure of finite
dimensional $G$-graded simple algebras over an algebraically closed
field $F$ of characteristic zero.

Let us recall their theorem.

\begin{theorem}[\cite{BSZ}]\label{Bahturin Zaicev Sehgal}

Let $B$ be a $G$-simple algebra. Then there exists a subgroup $H$ of
$G$, a $2$-cocycle $f:H\times H\longrightarrow F^{*}$ where the
action of $H$ on $F$ is trivial, an integer $r$ and an $r$-tuple
$(g_{s_1},\ldots,g_{s_r})\in G^{(r)}$ such that $B$ is
$G$-graded isomorphic to $C=F^{f}H\otimes M_{r}(F)$, where $C_g=
\span_F\{u_h \otimes e_{i,j}: g=g_{s_i}^{-1}hg_{s_j}\}$. Here
$F^{f}H= \sum_{h\in H}Fu_{g}$ is the twisted group algebra of $H$
over $F$ with the $2$-cocycle $f$ and $e_{i,j}\in M_{r}(F)$ is the
$(i,j)$-elementary matrix.

In particular the idempotents $1 \otimes e_{i,i}$ as well as the
identity of $B$ are homogeneous of degree $e\in G$.
\end{theorem}

\begin{remark}
Note that the grading above induces $G$-gradings on the subalgebras
$F^{f}H\otimes F$ and $F\otimes M_{r}(F)$ which are \textit{fine}
and \textit{elementary} respectively (see \cite{AHN}, \cite{BSZ_0},
\cite{BS_0} for the precise definitions).
\end{remark}

Let us analyze the $K$ component of a $G$-graded simple algebra
$B$ using the terminology of Theorem \ref{Bahturin Zaicev Sehgal}.

\begin{lemma} \label{double cosets}

Let $B$ as in Theorem \ref{Bahturin Zaicev Sehgal}. Then the following hold.

\begin{enumerate}

\item

A basis element $u_{h}\otimes e_{i,j}$ is in the $K$-homogeneous
component of $B$ if and only if $g_{s_i}^{-1}hg_{s_j}\in K$.

\item

Let $(i,j)$ be any pair of indices $1 \leq i,j\leq r$ $($in the
elementary grading of $M_{r}(F)$$)$. Then there exists a basis
element $u_{h}\otimes e_{i,j}$ in the $K$-homogeneous component of
$B$ if and only if $g_{s_i}$ and $g_{s_j}$ determine the same $H-K$
double coset in $G$.

\item

The relation ``~$i \sim j$ if and only if $g_{s_i}$ and $g_{s_j}$
determine the same $H-K$ double coset in $G$" is an equivalence
relation on the set of indices  $\{1,\ldots,r\}$.

\item

Fix $i$, $1\leq i \leq r$, and let $[i]$ be the equivalence class it
represents. Then the set of basis elements $u_{h}\otimes e_{j,k}$
which belong to the $K$-homogeneous component $B_{K}$ of $B$ such
that $[j]=[k]=[i]$ span a $K$-graded simple algebra. Furthermore, if
we denote this $K$-graded simple algebra by $B_{K,i}$, then
$B_{K,i}\cong F^{g_{s_i}(f)}(g_{s_i}^{-1}Hg_{s_i}\cap K)\otimes
M_{\pi_i}(F)$ where $\pi_i$ is the number of indices in
$\{1,\ldots,r\}$ which are equivalent to $i$ and $g_{s_i}(f)$ is the
action of $g_{s_i}$ on the $2$-cocycle $f$ given by
$g_{s_i}(f)(g_{s_i}^{-1}hg_{s_i},g_{s_i}^{-1}\widehat{h}g_{s_i})=f(h,\widehat{h})$
for any $h, \widehat{h}$ in $H$. In particular, the dimension of
$B_{K,i}$ is $|g_{s_i}^{-1}Hg_{s_i} \cap K|\pi_{i}^{2}$.
\item

Let $\Theta$ be a set of indices in $\{1,\ldots,r\}$ which represent the different equivalence classes.
Then there is an isomorphism of $K$-graded algebras $B_{K}\cong \oplus_{i\in \Theta}B_{K,i}$.

\end{enumerate}

\end{lemma}

\begin{proof} We prove (4) (1-3 and 5 are clear). Let $I$ be a
nonzero $K$-graded $2$-sided ideal in $B_{K,i}$. We need to show
$I=B_{K,i}$. Observe that the idempotents $1\otimes e_{j,j}$,
$[j]=[i]$, belong to the $e$-component of $B$ and hence are in
$B_{K}$. Now, if $z=\sum_{s,j,k}\alpha_{s,j,k}u_{h_{s,j,k}}\otimes
e_{j,k}$ is a non-zero element in $I$, multiplying from left and
right by suitable idempotents $1\otimes e_{j,j}$ we may assume that
$z=z_{j,k}= (\sum\alpha_{s}u_{h_{s}})\otimes e_{j,k}$. But for
different $h_{s}$ in $H$, the basis elements $u_{h_{s}}\otimes
e_{j,k}$ belong to different homogeneous components and so, since
$I$ is $K$-graded, we may assume $z=u_{h}\otimes e_{j,k}$. Take now
any basis element $u_{h^{'}}\otimes e_{u,v}$ in $B_{K,i}$. We need
to show it is in $I$. Indeed, since $[u]=[j]=[k]=[v]$, there exist
$h_{1}$ and $h_{2}$ in $H$ such that $u_{h_{1}}\otimes e_{u,j}$ and
$u_{h_{2}}\otimes e_{k,v}$ are in $B_{K,i}$. Multiplying $z$ from
left and right by these elements we obtain that $I$ contains a basis
element $u_{\widehat{h}}\otimes e_{u,v}$ for some $\widehat{h}\in
H$. Finally we multiply $u_{h^{'}}\otimes e_{u,v}$ from the right by
$(u_{\widehat{h}})^{-1}u_{h^{'}}\otimes e_{v,v}\in B_{K,i}$ and the
first part of (4) is proved. To see the second part of (4) we fix
$h_{t}\in H$ such that $g_{s_i}^{-1}h_{t}g_{s_t}\in K$ for every
$t\sim i$ (if $t=i$ we may take $h_t=e$). Then, any basis element
$u_{h}\otimes e_{t,k} \in~B_{K,i}$ can be written uniquely as
$\lambda u^{-1}_{h_{t}} u_{\bar{h}}u_{h_{k}}\otimes e_{t,k}$ where
$\lambda \in F^{*}$ and $g^{-1}_{s_i}\bar{h}g_{s_i}\in~
g^{-1}_{s_i}Hg_{s_i}\cap~K$. It follows that if we equip
$M_{\pi_i}(F)$ with the elementary $K$-grading given by the
$\pi_i$-tuple

$$
(e=~g_{s_i}^{-1}eg_{s_i},~
g_{s_i}^{-1}h_{j_2}g_{s_{j_2}},~g_{s_i}^{-1}h_{j_3}g_{s_{j_3}},~.~.~.~,
g_{s_i}^{-1}h_{j_{\pi_i}}g_{s_{j_{\pi_i}}})
$$
where $(i=~j_1,j_2, j_3,\ldots,j_{\pi_i})$ runs over all indices in
$(1,\ldots,r)$ which are equivalent to $i$, then the map

$$
u^{-1}_{h_{t}}u_{\bar{h}}u_{h_{k}}\otimes e_{t,k} \longmapsto
v_{g_{s_i}^{-1}\bar{h}g_{s_{i}}}\otimes e_{\mu,\nu} , ~~~~~~~
j_{\mu}=t, j_{\nu}=k
$$
is an isomorphism of $K$-graded algebras of $B_{K,i}$ with
$F^{g_{s_i}(f)}(g_{s_i}^{-1}Hg_{s_i}\cap K)\otimes~M_{\pi_i}(F)$.
Details are omitted.
\end{proof}

Consider now the decomposition $A \cong S_{G} \oplus J$ where
$S_{G}$ is a $G$-graded semisimple algebra and $J$ is the Jacobson
radical. Let $S_{G}=(S_{G})_1\oplus \ldots\oplus (S_{G})_q$ be the
decomposition of $S_{G}$ into the direct sum of its $G$-simple
components. In view of the above decomposition we will consider
\textit{homogeneous} elements that belong to $J$ (radical elements)
and basis elements of the form $u_h\otimes e_{i,j} \in F^{f}H\otimes
M_{r}(F)$ (semisimple elements) where $F^{f}H\otimes M_{r}(F)$ is a
$G$-simple component of $S_G$.

We can now proceed to prove Theorem \ref{Main}.

Let
$$\Lambda=z_{1}v_{1}z_{2}\cdots z_{n}v_{n}z_{n+1}$$
be a (nonzero) monomial of elements of $A$ that realizes the value of
$\exp_{conj}^G(A)$. By linearity we may assume that the $z$'s are semisimple (homogeneous) elements and the
$v$'s are radical (homogeneous). Note also that we may assume that the semisimple
elements belong to \textit{different} $G$-simple components. Indeed,
those that repeat may be ``swallowed'' by the radical elements.

Our goal (as in \cite{Al}) is to replace each one of the
semisimple elements $z_{t}$ by a certain monomial $E_{t}$
(products of homogeneous elements) which roughly speaking, is
sufficiently ``rich", namely it ``visits" different basis elements
in $(S_G)_t$ in a suitable way.

We describe the construction in four steps. Before we start, let us
recall a lemma which is well known.

\begin{lemma} \label{monomial}

Consider the $r\times r$ elementary matrices
$\{e_{i,j}\}_{i,j=1}^{r}$. Then there is a nonzero product of the
$r^{2}$ different elementary matrices of the form $e_{1,1}e_{1,2}e_{2,2} \cdots e_{i,1}$. Clearly the
total value is $e_{1,1}$.

Furthermore by decomposing this product into two parts and
switching them we obtain a nonzero monomial of that kind which
starts with $e_{i,i}$ for any $i=1,\ldots,r$.
\end{lemma}

\underline{Step 1} (construction of $\Lambda_1$)

Multiply each one of the $z_{t}~'s$  in $\Lambda$ by a primitive
idempotent $1\otimes e_{i,i}$ of the corresponding $G$-simple
component, say on the right, so that the product remains non zero
(clearly such idempotent exists by linearity). We denote these
idempotents by $x_{t}$, $t=1,\ldots,n+1$.

\underline{Step 2} (construction of $\Lambda_2$)

Replace each one of the idempotents $x_{t}$ ($ = 1\otimes e_{i,i}$
say) by a nonzero monomial $X_{t}$ of the form $1\otimes
e_{i,i}\times \cdots \times 1\otimes e_{j,i}$ (see~ Lemma
\ref{monomial}). Here, if the idempotent $x_{t}= 1\otimes e_{i,i}$
belongs to the $G$-simple component $(S_G)_t$ $(\cong F^{f}H\otimes
M_{r}(F))$ then the cardinality of $X_{t}$ is $r^2$.

\underline{Step 3} (construction of $\Lambda_3$)

Replace each one of the idempotents $a_{t,k}=1\otimes e_{k,k}$ in
$X_{t}$, $t=1,\ldots,n+1$, by a monomial $Y_{t,k}$ of the form

$$
u_{h_1}\otimes e_{k,k} \times a_{k} \times u_{h_2}\otimes e_{k,k}
\times \cdots \times u_{h_p}\otimes e_{k,k} \times a_{k}
$$
where $p=ord(H)$ and the set $\{h_1, h_1h_2,\ldots, h_1h_2\cdots
h_p\}$ runs over all elements of $H$.

We denote by $\widehat{X}_{t}$ the monomial obtained from $X_{t}$
and by $\Lambda_3$ the monomial obtained from $\Lambda_2$.

\underline{Step 4} (construction of $\widehat{\Lambda}=\Lambda_4$)

The objective of this step is nothing but to ``clean" the monomial

$$
\Lambda_3=z_{1}\widehat{X}_{1}v_{1}z_{2}\widehat{X}_{2}\cdots z_{n}\widehat{X}_{n}v_{n}z_{n+1}\widehat{X}_{n+1}
$$

Indeed, we throw $z_{1}$ from $\Lambda_3$ and replace
$v_{t}z_{t+1}$ by a new radical homogeneous element which we
denote again by $v_{t}$.

\begin{remark}

Clearly (by construction) the monomial

$$
\widehat{\Lambda}=\widehat{X}_{1}v_{1}\widehat{X}_{2}\cdots \widehat{X}_{n}v_{n}\widehat{X}_{n+1}
$$

is non zero and realizes $\exp_{G}^{Conj}(A)$.

\end{remark}

Our next task, roughly speaking, is to construct a certain subset of
$G$, denoted by $\Omega_{0}$, which will ``determine" different
decompositions of the monomial $\widehat{\Lambda}$. The set
$\Omega_{0}$ will play a decisive role in the proof of Theorem
\ref{Main}.

Write $\widehat{\Lambda}=b_1b_2b_3\cdots b_d$ where $b_l$ is
either a basis element of the form $u_h\otimes e_{i,j}$ which
appears in $\widehat{X}_{t}$, $t=1,\ldots,n+1$ or a radical
element $v_{t}$ for some $t=1,\ldots,n$.

Consider the subwords of $\widehat{\Lambda}$ of the form

$$
b_1,~b_1b_2,~b_1b_2b_3,~\ldots,~b_1b_2\cdots b_d.
$$

For each subword we consider its homogeneous degree in $G$ and we
let $\Omega$ be the set of elements in $G$ obtained in that way.
Clearly, any element in $g\in \Omega$ may be the homogeneous degree
of more than one subword of the form $b_1b_2\cdots b_j$ in
$\widehat{\Lambda}$ and so we let $\mu(g)\in \{1,\ldots,d\}$ be the
minimal length $j$ it appears. We refer to $\mu(g)$ as the length of
$g\in \Omega$ in  $\widehat{\Lambda}$. Next we consider the set
$\Pi$ of left $K$-cosets of $G$ that are represented by elements of
$\Omega$ and let $\Omega_{0}$ be the set of representatives for the
left $K$-cosets (in $\Pi$) of minimal length in $\widehat{\Lambda}$.

\begin{definition}
Let $b_jb_{j+1}\cdots b_{j+k}$ be a $G$-graded monomial and $S$ a nonempty subset of $G$.

\begin{enumerate}

\item

We say that $b_jb_{j+1}\cdots b_{j+k}$ is an \textit{$S$-monomial},
if its homogeneous degree belongs to $S$.

\item

We say that $b_jb_{j+1}\cdots b_{j+k}$ has no \textit{proper}
$S$-submonomial if either $k=0$ or else, there is \textit{no}
$S$-monomial of the form $b_jb_{j+1}\cdots b_{j+p}$ where $0\leq p <
k$. Note that still there may exist an $S$-monomial of the form
$b_{j+u}\cdots b_{j+q}$ where $1 < u \leq q \leq k$.
\end{enumerate}

\end{definition}

For any $g\in \Omega_{0}$ we consider the decomposition of

$$
\widehat{\Lambda} = X_{g}\Sigma_{K,1}\Sigma_{K,2}\cdots\Sigma_{K,d}Y(g)
$$
where
\begin{enumerate}

\item

$X_{g}$ is a $g$-monomial and has no $g$-submonomial.

\item

Each $\Sigma_{K,j}$ is a $K$-monomial and has no $K$-submonomial.

\item

$Y(g)$ is an homogeneous monomial (which depends on $g\in G$) and has no $K$-submonomial.

\end{enumerate}

\begin{remark}
Note that by the definition of the set $\Omega_{0}$, the above
decomposition of $\widehat{\Lambda}$ exists and is unique for any
$g\in \Omega_{0}$.

\end{remark}

Consider the decomposition of $\widehat{\Lambda}$ determined by $g
\in \Omega_{0}$. Clearly, the $g$-monomial $X_{g}$ and the
$K$-monomials $\Sigma_{K,j}$ may be multiplied from the right (and
give a nonzero product) by a unique primitive idempotent of the form
$1\otimes e_{i,i}$ which belongs to one of the $G$-simple components
$(S_{G})_{t}$ of $S$. We refer to these idempotents as the $K$-stops
determined by $g\in \Omega_{0}$ in $\widehat{\Lambda}$.

\begin{lemma} \label{accounts of visits I}
With the above notation we have:

\begin{enumerate}

\item

Each $g \in \Omega_{0}$ determines an ordered set
$\{\alpha_{g,1},\ldots, \alpha_{g,n_g}\}$ of $K$-stops in
$\widehat{\Lambda}$.

\item

Each $K$-stop $\alpha_{g,i}$ determines uniquely a $G$-simple
component $(S_{G})_{\alpha_{g,i}}$. Moreover it determines
uniquely a $K$-simple component $(S_{G})_{K,\alpha_{g,i}}$. We
refer to these $K$-simple components as the $K$-simple components
determined by $g \in \Omega_{0}$.

\item

If two different $K$-stops $\alpha_{g,i}$ and $\alpha_{g,j}$ (same
$g$) determine the same $G$-simple component then they determine
also the same $K$-simple component.

\item

Each $K$-simple component $($in any $G$-simple component$)$ is
determined by a $K$-stop $\alpha_{g,i}$ for some $g \in \Omega_{0}$.
\end{enumerate}

\end{lemma}

\begin{proof}

Parts (1) and (2) follow directly from the construction. Part (3)
follows easily from the fact that between two $K$-stops in the same
$G$-simple component we have only semisimple $K$-monomials. Hence
having $K$-stops which belong to different $K$-simple components
would yield a zero product.

For the proof of (4) we recall first that any $K$-simple component
in $F^{f}H\otimes M_{r}(F)$ is determined by an equivalence class of
indices $\{1,\ldots,r\}$ (see Lemma \ref{double cosets}). Fix a
$K$-simple component $[i]$ in a $G$-simple component, say
$(S_{G})_{t}\cong F^{f}H\otimes M_{r}(F)$. This says that there is a
primitive idempotent $1\otimes e_{j,j}$ in $(S_{G})_{t}$, with $j\in
[i]$, that can be inserted on the right of one of the $b's$, say
$b_{m}$ in $\widehat{\Lambda}=b_1b_2b_3\cdots b_d$ and yield a non
zero product. Let $g \in G$ be the homogeneous degree of the
monomial $b_1b_2b_3\cdots b_m$. Clearly $g \in \Omega$ and hence
$gK$, the left $K$-coset represented by $g$, intersects $\Omega$
non-trivially. Observe that by the definition of $\Omega$, any
element $g{'} \in gK \cap \Omega$ determines a $K$-stop in
$\widehat{\Lambda}=b_1b_2b_3\cdots b_d$ and hence if $\mu(g)\leq
\mu(g{'})$ (the lengths in $\widehat{\Lambda}$) we have that any
$K$-stop in $\widehat{\Lambda}=b_1b_2b_3\cdots b_d$ determined by
$g{'}$ is determined also by $g$. We conclude that if $g$ is the
representative $gK \cap \Omega$ of minimal length in
$\widehat{\Lambda}$ (hence in $\Omega_{0}$) it determines the
$K$-simple component represented by $[i]$.

\end{proof}

The next lemma is key. It computes the number of elements in
$\Omega_{0}$ which determine the different $K$-simple components.

Fix $g$ in $\Omega_{0}$ and let

$$
\widehat{\Lambda} =
X_{g}\Sigma_{K,1}\Sigma_{K,2}\cdots\Sigma_{K,j}\cdots\Sigma_{K,d}Y(g)
$$
be the corresponding decomposition of $\widehat{\Lambda}$. Consider the $K$-stop
$\alpha_{g,j}$ determined by
$X_{g}\Sigma_{K,1}\Sigma_{K,2}\cdots\Sigma_{K,j}$.

By construction, $\alpha_{g,j}$ is a primitive idempotent of the
form $1\otimes e_{i,i}$, $1\leq i \leq r$, in the $t$-th $G$-simple
component $(S_{G})_{t}\cong F^{f}H\otimes M_{r}(F)$. Let $k_{1}\in
K$ be the homogeneous degree of the monomial
$\Sigma_{K,1}\Sigma_{K,2}\cdots\Sigma_{K,j}$. In the next lemma we use the same notation as above (in
particular $(g_{s_1},\ldots,g_{s_r})$ is the $r$-tuple in $G^{(r)}$ which determines
the elementary $G$-grading on $M_{r}(F)$).

\begin{lemma} \label{account of visits II}

The following hold.

\begin{enumerate}
\item

Any element in $gk_{1}g_{s_i}^{-1}Hg_{s_i}K$ \textit{which
determines} a $K$-simple component in the $t$-th $G$-simple
component $(S_{G})_{t}$, determines the same $K$-simple component as
$g$ does.

\item

Any $K$-left coset which is contained in
$gk_{1}g_{s_i}^{-1}Hg_{s_i}K$, is represented in $\Omega$ and hence
is represented in $\Omega_{0}$.

\item

Any $\widehat{g}\in G$ that determines the same $K$-simple component
in the $t$-th $G$-simple component as $g$ does, is in
$gk_{1}g_{s_i}^{-1}Hg_{s_i}K$.

\end{enumerate}
\end{lemma}

\begin{proof}

Recall first that the homogeneous degree of a semisimple element
$b_l=u_{h}\otimes e_{i,j}$ is $g_{s_i}^{-1}hg_{s_j}$. Suppose
$\widehat{g}=gk_{1}g_{s_i}^{-1}h_{1}g_{s_i}k_2$ is an element in
$gk_{1}g_{s_i}^{-1}Hg_{s_i}K$ which determines a $K$-stop in the
$t$-th $G$ simple component $(S_{G})_{t}$ with primitive idempotent
$1\otimes e_{j,j}$. We need to show that $j\sim i$.
Indeed, by construction there is $k_3$ in $K$ and $h_2$ in $H$
such that $\widehat{g}k_3=gk_1g_{s_i}^{-1}h_2g_{s_j}$ that is

$$
gk_{1}g_{s_i}^{-1}h_1g_{s_i}k_2k_3=gk_1g_{s_i}^{-1}h_2g_{s_j}.
$$

\smallskip
This shows that $h_1g_{s_i}k_2k_3=h_2g_{s_j}$ and $j\sim i$ as
desired.

Let us prove (2). We are assuming that $g$ determines $1\otimes
e_{i,i}$ in the $t$-th $G$-simple component. Consider the set
$gk_{1}g_{s_i}^{-1}Hg_{s_i}$. Clearly it contains a set of
representatives for the left $K$-cosets which are contained in
$gk_{1}g_{s_i}^{-1}Hg_{s_i}K$. On the other hand, by the
construction of the monomial $\widehat{\Lambda}$ (step 3), each
element in $gk_{1}g_{s_i}^{-1}Hg_{s_i}$ determines a $K$-stop
$1\otimes e_{i,i}$ and the result follows.

To prove (3) suppose $\widehat{g}\in G$ determines the same
$K$-simple component as $g$ does in the $t$-th $G$-simple component.
This means that $\widehat{g}$ determines a $K$-stop $1\otimes
e_{j,j}$ in the $t$-th simple component and $j\sim i$. It follows
that there exist elements $k_1$ and $k_2$ in $K$ and $h$ in $H$ such
that

$$
\widehat{g}k_2=gk_1g_{s_i}^{-1}hg_{s_j}.
$$

\smallskip

But from the condition on relation of $j$ and $i$ we know that
$g_{s_j}$ and $g_{s_i}$ represent the same $H-K$ double coset of
$G$ and hence there exist $k_3$ in $K$ and $h_2$ in $H$ such that
$g_{s_j}=h_2g_{s_i}k_3$. This implies

$$
\widehat{g}k_2=gk_1g_{s_i}^{-1}hh_2g_{s_i}k_3
$$
and hence $\widehat{g}\in gk_{1}g_{s_i}^{-1}Hg_{s_i}K$ as desired.
This completes the proof of the lemma.
\end{proof}

\begin{remark}
We emphasize that parts (1) and (3) of the lemma are not exactly
opposite to each other for it may be the case that an element in
$gk_{1}g_{s_i}^{-1}Hg_{s_i}K$ does not determine any $K$-simple
component in $(S_{G})_{t}$. This is the reason that we couldn't get
through with the set $\Omega$ and we needed a finer set, namely a
set of left $K$-cosets representatives $\Omega_{0}$.
\end{remark}

From our construction we see that any $g\in \Omega_{0}$ determines a
unique $K$-simple component in some of the $G$-simple components
$(S_{G})_{1},\ldots,(S_{G})_{n+1}$. We denote the $K$-simple
components by $B_{g,K,1},\ldots,B_{g,K,n+1}$ where the parameter $g$
says that the $K$-components depend on $g$. We put $B_{g,K,j}=0$ for
those $G$-simple component which are not visited in the
decomposition of $\widehat{\Lambda}$ determined by $g$.

\begin{corollary} \label{number of visits}

Consider all decompositions of the form

$$
\widehat{\Lambda} =
X_{g}\Sigma_{K,1}\Sigma_{K,2}\cdots\Sigma_{K,j}\cdots\Sigma_{K,d}Y(g)
$$
where $g\in \Omega_{0}$.

Take any $K$-simple component, say the $K$-simple component $[i]$ in
the $t$-th $G$-simple components $(S_{G})_{t}$. Let $N_{S_{G}}([i])$
be the number of such decompositions in which the $K$-simple
component $[i]$ is ``visited" (that is the decompositions that
determine a $K$-stop $1\otimes e_{j,j}$ in $(S_{G})_{t}$ and $j\sim
i$). Then

$$
N_{S_{G}}([i])=|Hg_{s_i}K|/|K|
$$

\end{corollary}

\begin{proof}

By Lemma \ref{accounts of visits I} part (4), we know that the
$K$-simple component [i] is ``visited" in the decomposition of
$\widehat{\Lambda}$ determined by some element $g\in \Omega_{0}$.
Then by Lemma \ref{account of visits II} we conclude that the number
is

$$
N_{S_{G}}([i])=|gk_1g_{s_i}^{-1}Hg_{s_i}K|/|K|=|Hg_{s_i}K|/|K|.
$$

\end{proof}

We are now ready to complete the proof of Theorem \ref{Main}.
Let us assume

$$
\exp^{Conj}_{G}(A) > [G:K]^2 \exp^{Conj}_{K}(A_{K}).
$$

and get a contradiction.

Recall we are assuming that the monomial

$$
\widehat{\Lambda}=\widehat{X}_{1}v_{1}\widehat{X}_{2}\cdots \widehat{X}_{n}v_{n}\widehat{X}_{n+1}
$$

realizes $\exp^{Conj}_{G}(A)$. Consider the different decompositions

$$
\widehat{\Lambda} =
X_{g}\Sigma_{K,1}\Sigma_{K,2}\cdots\Sigma_{K,j}\cdots\Sigma_{K,d}Y(g)
$$

where $g\in \Omega_{0}$. By the definition of $\exp^{Conj}_{K}(A_{K})$ we have that

$$
\exp^{Conj}_{K}(A_{K}) \geq dim_{F}(B_{g,K,1})+ \cdots +
dim_{F}(B_{g,K,n+1})
$$

for every $g\in \Omega_{0}$. Consequently for every $g\in \Omega_{0}$

$$
\exp^{Conj}_{G}(A) > [G:K]^2 \sum_{j=1}^{n+1} dim_{F}(B_{g,K,j}).
$$

Summing over all elements of $\Omega_{0}$ we have

$$
\sum_{g\in \Omega_{0}} \exp^{Conj}_{G}(A) > [G:K]^2 \sum_{g\in \Omega_{0}} \sum_{j=1}^{n+1} dim_{F}(B_{g,K,j}).
$$

Let us show that this is not the case, that is

$$
\sum_{g\in \Omega_{0}} \exp^{Conj}_{G}(A) \leq [G:K]^2 \sum_{g\in \Omega_{0}} \sum_{j=1}^{n+1} dim_{F}(B_{g,K,j})
$$

$$
=[G:K]^2 \sum_{j=1}^{n+1} \sum_{g\in \Omega_{0}} dim_{F}(B_{g,K,j}).
$$

Now by construction,
$$
\exp^{Conj}_{G}(A)=\sum_{j=1}^{n+1} dim_{F}((S_{G})_{j})
$$
and so, we may prove the above inequality by showing it holds for each $G$-simple component $(S_{G})_{t}$, $t=1,\ldots, n+1$, separately. Thus we will show that

$$
\sum_{g\in \Omega_{0}} dim_{F}(S_{G})_{t} \leq [G:K]^2 \sum_{g\in \Omega_{0}}dim_{F}(B_{g,K,j}).
$$

To see this we adopt the notation of Theorem \ref{Bahturin Zaicev
Sehgal} and of Lemma \ref{double cosets}, namely

$$
(S_{G})_{t}\cong F^{\alpha}H \otimes M_{r}(F)
$$
where $H$ is a subgroup of $G$, $(g_{s_1},\ldots,g_{s_r})\in
G^{(r)}$ is the $r$-tuple which determines the elementary grading on
$M_{r}(F)$, $\{\Phi_1,\ldots,\Phi_m\}$ is the set of the double
$H-K$-cosets in $G$ and $\pi_j$, $j=1,\ldots,m$, is the number of
indices $i$ in $\{1,\ldots,r\}$ such that $g_{s_i}$ represents
$\Phi_j$. Note that by our notation we allow $\pi_j=0$, this would
mean that $\Phi_j$ is not represented by any element in
$(g_{s_1},\ldots,g_{s_r})$.

Now, by Corollary \ref{number of visits}, each (nonzero) $K$-simple
component $B_{g,K,j}$ appears in the sum above precisely
$|Hg_{s_i}K|/|K|$ times where $i$ is the index that determines
$B_{g,K,j}$ (see~Lemma \ref{double cosets}). Furthermore (by the
same lemma, part 4), its dimension is $|g_{s_i}^{-1}Hg_{s_i}\cap
K|\pi_{j}^{2}$.

In view of the paragraph above we need to prove the inequality

$$
| H | (\pi_1 + \cdots + \pi_m)^{2} \leq [G:K]\sum_{j=1}^{m}(|
Hg_{s_i}K |/| K |) | g_{s_i}^{-1}Hg_{s_i}\cap K | \pi_{j}^{2}.
$$

But one knows that

$(| Hg_{s_i}K |/| K |) | g_{s_i}^{-1}Hg_{s_i}\cap K | = | H |$ and hence we need to show that

$$
| (\pi_1 + \cdots + \pi_m)^{2} \leq [G:K] \sum_{i=1}^{m}\pi_{i}^{2}.
$$

As $m$ is the number of double $H-K$-cosets in $G$, we have that $m
\leq [G:K]$, and hence the result will follow if we prove

$$
(\pi_1 + \cdots + \pi_m)^{2} \leq m \sum_{i=1}^{m}\pi_{i}^{2}.
$$

But this is equivalent to  $\sum_{i<j}(\pi_{j}-\pi_{i})^2 \geq 0$
and Theorem \ref{Main} is proved.

\end{section}

\end{document}